\documentclass[a4paper,11pt,twoside,reqno]{amsart}

\usepackage[english]{babel}
\usepackage[utf8]{inputenc}

\usepackage[margin=3.2cm]{geometry}
\usepackage{enumitem}
\usepackage{bold-extra}
\usepackage{ mathrsfs }

\usepackage[dvipsnames]{xcolor}
\usepackage[
 colorlinks=true,
 linkcolor=MidnightBlue,
 citecolor=MidnightBlue,
 urlcolor=MidnightBlue
]{hyperref}

\usepackage{amsmath}
\usepackage{amsfonts}
\usepackage{amssymb}
\usepackage{amsthm}
\usepackage{comment}
\usepackage{mathtools}

\newtheorem{theorem}{Theorem}[section]
\newtheorem{corollary}[theorem]{Corollary}

\newtheorem{lemma}[theorem]{Lemma}

\theoremstyle{definition}

\theoremstyle{remark}
\newtheorem{remark}[theorem]{Remark}

\numberwithin{equation}{section}
\newcommand\eps{\varepsilon}
\newcommand\ve{\varepsilon}

\newcommand\R{\mathbb{R}}
\newcommand\Z{\mathbb{Z}}
\newcommand\ZZ{\mathbb{Z}}
\newcommand\ZZp{\mathbb{Z}_{\text{prim}}}
\newcommand\supp{\operatorname{supp}}
\newcommand\ord{\operatorname{ord}}

\newcommand\NN{\mathbb{N}}

\newcommand\Q{\mathbb{Q}}
\newcommand\QQ{\mathbb{Q}}
\newcommand\PP{\mathbb{P}}
\newcommand\K{\overline{\mathbb{Q}}}

\DeclareMathOperator\rad{rad}
\DeclareMathOperator\rank{rank}
\DeclareMathOperator\Res{Res}
\renewcommand\div{\mathrm{div}}

\renewcommand{\leq}{\leqslant}
\renewcommand{\le}{\leqslant}
\renewcommand{\geq}{\geqslant}
\renewcommand{\ge}{\geqslant}
\DeclarePairedDelimiter\abs{\lvert}{\rvert}

\title{Sums of three powerful numbers}
\author{Tim Browning}
	\address{IST Austria\\
		Am Campus 1\\
		3400 Klosterneuburg\\
		Austria}
	
	\email{tdb@ist.ac.at\\
		matteo.verzobio@gmail.com}	
	\author{Matteo Verzobio}
\date{}

\begin{document}

\begin{abstract}
Let $p,q,r\geq 2$ and consider the Campana orbifold
\[
\left(
 \PP^1,
 \left(1-\tfrac1p\right)[0]
 +\left(1-\tfrac1q\right)[1]
 +\left(1-\tfrac1r\right)[\infty]
\right).
\]
Primitive positive Campana points on this orbifold correspond to
solutions of $a+b=c$ in which $a$, $b$, and $c$ are respectively
$p$-full, $q$-full, and $r$-full. We establish upper
bounds for the number of such points of bounded height in a broad
range of exponents, with a power-saving over the trivial bound. 
The main analytic input is an estimate
for primitive integral points in lopsided boxes on generalized
Fermat surfaces
\[
a_1x^p+a_2y^q+a_3z^r=0,
\]
which is uniform in the coefficients.
\end{abstract}

\maketitle

\thispagestyle{empty}
 \setcounter{tocdepth}{1}
 \tableofcontents

\section{Introduction}

The arithmetic of Campana points provides a natural framework
interpolating between the study of rational points and that of
integral points on algebraic varieties. 
In the log-Fano setting, 
a Manin-type conjectural framework for the density of Campana points has been worked out by 
Pieropan--Smeets--Tanimoto--V\'arilly-Alvarado~\cite{c-manin}, with 
recent refinements by 
Chow--Loughran--Takloo-Bighash--Tanimoto 
\cite{chow}. 
In this paper we consider one of the simplest, albeit 
highly
nontrivial, examples. For integers $p,q,r\geq2$, define the $\QQ$-divisor
\[
\Delta_{p,q,r}
 =
 \left(1-\frac1p\right)[0]
 +\left(1-\frac1q\right)[1]
 +\left(1-\frac1r\right)[\infty]
\]
on $\PP^1$. Then 
$(\PP^1,\Delta_{p,q,r})$ is an orbifold 
in the sense of Campana~\cite{campana}.

Recall that a nonzero integer $n$ is called $m$-full if
$v_\ell(n)=0$ or $v_\ell(n)\geq m$, 
for every prime $\ell$. 
For $a,b,c\in\NN$, 
a primitive solution to 
$a+b=c$ is automatically pairwise coprime. The map
$(a,b,c)\mapsto (a:c)\in\PP^1(\QQ)$
then identifies primitive solutions for which $a$
is $p$-full, $b$ is $q$-full, and $c$ is $r$-full 
with the Campana points of
 $(\PP^1,\Delta_{p,q,r})$.
Indeed, at
each finite prime $\ell$, the local intersection multiplicities with
the divisors $[0]$, $[1]$, and $[\infty]$ are, respectively, 
$v_\ell(a), v_\ell(b)$, and $v_\ell(c)$.
Let $\mathcal S_m$ denote the set of positive $m$-full integers and
write
\begin{equation}\label{eq:hawk}
N(B)
 =
 \#\left\{
 (a,b,c)\in
 (\mathcal S_p\times\mathcal S_q\times\mathcal S_r)\cap[1,B]^3:
 \gcd(a,b,c)=1,\ a+b=c
 \right\}.
\end{equation}
Since $a,b>0$ and $a+b=c$, the usual height of the corresponding
point $(a:c)$ is $c$. Thus $N(B)$ counts the positive Campana points
on $(\PP^1,\Delta_{p,q,r})$ of height at most $B$.

The geometry of the orbifold is governed by
the divisor 
$K_{\PP^1,\Delta_{p,q,r}}=K_{\PP^1}+\Delta_{p,q,r}$, with degree
\[
\deg(K_{\PP^1,\Delta_{p,q,r}})
 =
 1-\left(\frac1p+\frac1q+\frac1r\right).
\]
If $-K_{\PP^1,\Delta_{p,q,r}}$ is ample then 
 $(\PP^1,\Delta_{p,q,r})$ is log-Fano and subject to the conjectures in~\cite{chow,c-manin}.
This is the case when $p=q=r=2$, for example, in which case 
it is conjectured that $N(B)\sim c B^{1/2}$, as $B\to \infty$, for a suitable constant $c>0$. 
The best result we have in this direction is the upper bound 
\[
N(B)\ll_\delta B^{3/5-\delta},
\] 
for any $\delta<\frac{3}{1555}$, which is due to Heath-Brown~\cite{HB-squarefull} and which refines earlier work of 
Browning and Van Valckenborgh~\cite{squareful}. 
In the case $p=q=r$, 
as a direct extension of the latter work, 
the balanced dyadic range suggests the exponent
\[
\frac{6(r-1)}{3r^2-r}
=
\frac2r-\frac{4}{r(3r-1)},
\]
although 
controlling the lopsided ranges would require additional
input.

In this note we place ourselves in the log-general type range, where
\begin{equation}\label{eq:gen}
\frac1p+\frac1q+\frac1r<1.
\end{equation}
As explained by Abramovich and Várilly-Alvarado~\cite[Conjecture 1.2]{alvarado}, 
it follows from conjectures of 
Campana that the corresponding Campana points
are not Zariski dense in $\PP^1$ in this case, so that $N(B)=O_{p,q,r}(1)$.
This prediction is also consistent with
the $abc$ conjecture: if $(a,b,c)$ is a solution counted by $N(B)$, 
note that $\rad(abc)\leq a^{1/p}b^{1/q}c^{1/r}\leq c^{1/p+1/q+1/r}$.
But then 
\[
c\ll_\varepsilon \rad(abc)^{1+\varepsilon}
 \leq c^{(1/p+1/q+1/r)(1+\varepsilon)},
\]
for any $\ve>0$. 
Since $(1/p+1/q+1/r)(1+\varepsilon)<1$ 
 whenever 
\eqref{eq:gen} holds, 
if $\ve>0$ 
is taken to be small enough, 
this bounds $c$.

Since such a finiteness result seems to be out of reach presently, 
our goal is to obtain quantitative evidence by proving upper 
bounds for $N(B)$, with 
a focus on the most difficult case, in which $p\geq q\geq r\geq 2$ are all distinct.
Since
$\#\mathcal S_m\cap[1,B]=O_m(B^{1/m})$, we always have the
estimate 
\begin{equation}\label{eq:triv}
N(B)\ll_{p,q}B^{1/p+1/q},
\end{equation}
if $p\geq q\geq r$, 
which we refer to as the trivial bound. 
In Theorem~\ref{prop:etagen}
we shall establish explicit criteria under which this bound admits a power
saving. One clean consequence is the following result.

\begin{theorem}\label{t:uv-main}
Let $u\geq v\geq0$ be integers and take 
$p=r+u$ and $q=r+v$.
Then, for all sufficiently large $r$, there exists an explicit
$\eta_{u,v}(r)>0$ such that
\[
N(B)
 \ll_{\varepsilon,u,v,r}
 B^{1/p+1/q-\eta_{u,v}(r)+\varepsilon},
\]
where
\[
\eta_{u,v}(r)
 =
 \frac1{r^2}+O_{u,v}(r^{-5/2}).
\]
\end{theorem}

Note that the case $p=q=r$ corresponds to taking $u=v=0$ in this result, in which case the argument behind the result actually yields $N(B) \ll_{\varepsilon,r} B^{2/r-\eta_r+\ve}$, where $\eta_r>0$ and satisfies
$\eta_r=\frac1{r^2}+O(r^{-5/2}) $. Thus our result gives a uniform power saving even in the fully symmetric case, 
although it does not quite reach the balanced exponent suggested above.
Our work is particularly effective in non-symmetric cases where
$p,q,r$ have a similar size. In the special case 
$(p,q,r)=(r+2,r+1,r)$, for example, we shall show in  Remark \ref{rem:rem'} that our method beats the trivial bound for all $r\geq 2$.

\medskip

The main analytic input is a uniform estimate for integral points on
generalized Fermat surfaces in rectangular boxes. Let
$a_1,a_2,a_3$ be nonzero integers and put
\begin{equation}\label{eq:f}
f(x,y,z)=a_1x^p+a_2y^q+a_3z^r.
\end{equation}
For $X,Y,Z\geq2$, define
\begin{equation}\label{eq:hen}
S(X,Y,Z,f)
 =
 \left\{
 (x,y,z)\in\ZZp^3:
~|x|\leq X,~ |y|\leq Y,~ |z|\leq Z,~
 f(x,y,z)=0
 \right\},
\end{equation}
where $\ZZp^3=\{(x,y,z)\in \ZZ^3: \gcd(x,y,z)=1\} $.
Let 
\begin{equation}\label{eq:W}
	W=\exp\left(\sqrt{\frac{\log X \log Y}{r}}\right).
\end{equation}
We shall prove the following result. 

\begin{theorem}\label{thm:main2}
	Let $\eps>0$ and let 
	 $f$ be given by~\eqref{eq:f} for 
	 $p,q,r\geq 2$
	 and $a_1,a_2,a_3\in \ZZ_{\neq 0}$.
	 Then 
\[
		\#S(X,Y,Z,f)\ll_{\ve,p,q,r} (XYZ)^{\ve}\left(W^2+W\max\{X,Y\}^{\frac{2}{\sqrt{\max\{p,q,36\}}}}+W\max\{X,Y\}^{\frac{1}{r}}\right).
\]
\end{theorem}

In this result, the exponent $r$ and the variable $z$ are distinguished, but one gets corresponding estimates 
under any permutation of the three pairs $(x,p)$, $(y,q)$ and $(z,r)$.
An important feature of this result is that the implied constant does not depend on the coefficients $a_1,a_2,a_3$ of $f$. 

Let $B=\max\{X,Y,Z\}$. Then
$W\leq B^{1/\sqrt{r}}$ in~\eqref{eq:W}, and Theorem~\ref{thm:main2} implies
\[
\#S(X,Y,Z,f)\ll_{\ve,p,q,r}
B^{\frac{1}{\sqrt{p}}+\frac{1}{\sqrt{q}}+\frac{1}{\sqrt{r}}+\ve}.
\]
Thus, one of the main features of our work is that
$\#S(X,Y,Z,f)$ can be bounded by a power of $B$ whose exponent
tends to zero as $p$, $q$, and $r$ jointly tend to infinity.
(In fact it is enough that any two of the exponents  tend to infinity.)
Moreover, the assumption $a_1a_2a_3\neq0$ is clearly necessary,
since otherwise $\#S(X,Y,Z,f)$ can be of order $B$.

If one does not care about uniformity, then in the range~\eqref{eq:gen}, it follows from work of Darmon and Granville~\cite[Theorem 2]{darmon} that the equation $f(x,y,z)=0$ has only finitely many primitive integral solutions.
In the range $1/p + 1/q + 1/r > 1$, 
it follows from work of Beukers~\cite{beukers} that the equation 
$f(x,y,z)=0$ has 
infinitely many nontrivial primitive solutions as soon as it has at least one. 
Recently, Arango-Pi\~neros~\cite{arango-counting}
used an alternative height ordering to 
 obtain an asymptotic formula for the number of primitive solutions in the range $1/p+1/q+1/r>1$. Finally, when $1/p + 1/q + 1/r = 1$, primitive solutions correspond to rational points on certain curves of genus one, and their behaviour depends on the associated Mordell–Weil group~\cite[Section 6]{darmon}. All these results are non-uniform, whereas Theorem~\ref{thm:main2} does not depend on the coefficients of $f$.

There are several other ways that one can bound $\#S(X,Y,Z,f)$ uniformly in the coefficients. For example, one can fix the $z$-variable and then bound the points on the remaining affine curve
using Heath-Brown~\cite[Theorem 15]{HeathBrown}. This yields
\[
\#S(X,Y,Z,f)\ll_{\ve,p,q,r}(XY)^{\ve}Z\min\{X^{1/q},Y^{1/p}\}.
\]
Alternatively, one can use basic Fourier analysis, as in 
Bernert--Browning--Lichtman--Ter\"av\"ainen 
\cite[Proposition 3.1]{abc}. This would yield a bound of the form 
\[
\#S(X,Y,Z,f)\ll_{\ve,p,q,r}(XYZ)^{1/2+\ve}. 
\]
Theorem~\ref{thm:main2}
 is stronger when $p$, $q$, and $r$ are not too small.
Our proof of it combines Salberger's determinant method~\cite{saldet}
 with function field
Diophantine geometry. The determinant method reduces the problem to
counting points on curves
\[
f(x,y,z)=c(x,y)=0.
\]
Writing $C$ for the normalization of $c(x,y)=0$, the curve $f=c=0$
is geometrically reducible only when
$a_1x^p+a_2y^q$
is a proper power in $\overline{\QQ}(C)$. The Brownawell--Masser
theorem then forces the plane degree of $C$ to be large compared with
$p$ and $q$, making applications of Heath-Brown~\cite{HeathBrown} and 
Bombieri--Pila~\cite{BombieriPila} 
particularly effective. 

In the case $X=Y=Z=B$, the bound can be slightly improved in several aspects: (i) following Castryck--Cluckers--Dittmann--Nguyen~\cite{nolog}, one can replace the $B^{\ve}$ factor with a small power of $\log B$;
(ii)
 following Ellenberg--Venkatesh~\cite{ellenberg}, one can obtain a small power 
 saving in the height 
 $\max\{\abs{a_1},\abs{a_2},\abs{a_3}\}$
 of $f$; and (iii) 
 following Binyamini--Cluckers--Kato~\cite{quadratic}, the precise dependence of these results
 on $\max\{p,q,r\}$ can be tracked.
 
\section{Counting points on codimension one slices}

Recall that $X,Y,Z\geq 2$ and let $B=\max\{X,Y,Z\}$. 
Our main task will be to estimate the quantity
\begin{equation}\label{eq:duck}
S(X,Y,Z,f,c)
 =
 \left\{
 (x,y,z)\in\ZZp^3:
 \begin{array}{l}
|x|\leq X,~ |y|\leq Y,~ |z|\leq Z,\\
 f(x,y,z)=c(x,y)=0
 \end{array}
 \right\},
\end{equation}
for given $c\in \Z[x,y]$. Indeed, once
combined with the determinant method, 
the following result will prove to be the principal ingredient in the proof of Theorem~\ref{thm:main2}.

\begin{theorem}\label{thm:main}
Let $D\geq 1$ and let $\eps>0$. Let $p,q,r\geq 2$ satisfy $p,q,r\leq D$, let 
$f$ be given by~\eqref{eq:f} for nonzero $a_1,a_2,a_3\in \ZZ$. 
Let $c(x,y)\in\Z[x,y]$ be of degree $d\leq D$ and irreducible over $\Q[x,y]$.
If $d\in \{1,2\}$ then
\begin{equation}\label{eq:main-bound2}
\#S(X,Y,Z,f,c)\ll_{\ve,D} B^{\ve}\max\{X,Y\}^{\frac{1}{r}}.
\end{equation}
If $d\geq 3$ then 
\begin{equation}\label{eq:main-bound}
\#S(X,Y,Z,f,c)\ll_{\ve,D} \max\{X,Y\}^{\frac{1}{m(p,q)}+\ve}+B^{\ve}\max\{X,Y\}^{\frac{1}{dr}},
\end{equation}
where
\begin{equation}\label{eq:mpq}
m(p,q)=\frac{\sqrt{\max\{p,q\}}}{2}.
\end{equation}
\end{theorem}

This result is completely uniform in the coefficients of $c$ and $f$.
Moreover, as discussed in Remark \ref{rem:rem}, a small improvement is possible when $r\geq 3$.
After dividing $(a_1,a_2,a_3)$ by their common divisor, we may and shall assume that $\gcd(a_1,a_2,a_3)=1$ in all that follows. 
Replacing $c$ by its primitive part, which does not change its
zero set, we may also assume that $c$ is primitive.

\subsection{Preliminary tools}\label{s:goat}

First, we shall use the Bombieri--Pila bound for integral points on plane curves~\cite[Theorem 5]{BombieriPila}.
\begin{lemma}\label{lem:BP}
Let $\ve>0$ and let $g\in\Z[U,V]$ be absolutely irreducible of degree $e$. Then 
\[
\#\{(u,v)\in\Z^2: |u|,|v|\leq B,\ g(u,v)=0\}\ll_{\eps,e}B^{\frac 1e+\eps}.
\] 
\end{lemma}

The following result is due to Vaughan and Wooley~\cite[Lemma~3.5]{VW}.

\begin{lemma}\label{lem:VW} 
Let $\ve>0$ and let $a,b,c,P$ be nonzero integers, with $P\geq2$. Then
 \[ \#\left\{ (u,v)\in\ZZ^2: |u|,|v|\leq P,\quad au^2+bv^2=c \right\} \ll_\varepsilon (|abc|P)^\varepsilon. \] 
 \end{lemma} 

For any polynomial $g$ with integer coefficients, the height of $g$ is 
denoted $H(g)$ and is 
defined to be the maximum modulus of its coefficients.
We shall use the following result of Heath-Brown~\cite[Theorem 5]{HeathBrown}.

\begin{lemma}\label{lem:large}
If $g\in\Z[U,V]$ is primitive, absolutely irreducible, and of degree $e$, then either the curve $g=0$ has $O_e(1)$ integer points of height at most $B$, or $H(g)\leq B^{O_e(1)}$. 
\end{lemma}

Third, we shall use Capelli's criterion for binomials~\cite[Theorem VI.9.1]{lang}. 
\begin{lemma}\label{lem:capelli}
If $K$ is a field of characteristic zero and $\alpha\in K^*$, then $T^r-\alpha$ is reducible over $K[T]$ only if $\alpha\in K^\ell$ for some prime $\ell\mid r$, or if $4\mid r$ and $\alpha\in -4K^4$.
\end{lemma}

We will always apply Lemma~\ref{lem:capelli} for $K$ such that $\K\subseteq K$, so that $-4$ is 
a square in $K$. Thus we can conclude that $\alpha\in K^2$ in the exceptional case. 
 Thus, if $T^r-\alpha$ is reducible over $K[T]$, this will always imply that $\alpha\in K^\ell$ for some prime $\ell\mid r$. Fourth, we need the Brownawell--Masser theorem on vanishing sums of $S$-units in function fields~\cite[Corollary~1]{BrownawellMasser}, which is the function field analogue of the $abc$ conjecture.

\begin{lemma}\label{lem:Sunit}
 Let $C$ be a smooth projective curve over $\K$ of genus $g$, let $K=\K(C)$, and let $u,v\in K^*$ be nonconstant $S$-units satisfying $1+u+v=0$. If no proper subsum vanishes, then $H_C(1:u:v)\leq 2g-2+|S|$. Here $H_C(1:u:v)$ is the degree of the corresponding map from $C$ to $\mathbb{P}^2$.
\end{lemma}

We restate the Brownawell–Masser theorem for polynomials, which is the Mason--Stothers theorem~\cite[Theorem 1]{polabc}.

\begin{lemma}\label{lem:abc}
Let $A,B,C\in \overline{\Q}[t]$ be nonzero, pairwise coprime polynomials satisfying \[
A+B+C=0,
\]
with at least one of $A,B,C$ not constant. Then
\[\max\{\deg A,\deg B,\deg C\}\leq \deg\rad(ABC)-1,\]
where $\rad(ABC)$ denotes the product of the distinct linear factors dividing $ABC$.
\end{lemma}

We use the following elementary fact several times. If a plane curve defined over $\Q$ is irreducible over $\Q$ but not over $\K$, then it contains $O_D(1)$ rational points when its degree is at most $D$. Indeed, every rational point lies in the intersection of two distinct Galois-conjugate components, and we can apply Bézout's theorem.

Finally, we notice that $f(x,y,z)=a_1x^p+a_2y^q+a_3z^r\in \Z[x,y,z]$ is absolutely irreducible. Indeed, each irreducible factor of $a_1x^p+a_2y^q\in \overline{\Q}[x,y]$ occurs with multiplicity one, and so Lemma~\ref{lem:capelli} implies that $f(x,y,z)$ is absolutely irreducible in 
$\overline{\Q}[x,y][z]$. Thus Gauss' lemma gives 
irreducibility in $\overline{\Q}[x,y,z]$.

\subsection{Easy cases}
Let $c(x,y)\in \Z[x,y]$ be absolutely irreducible of degree $d$. Let $C$ be the smooth projective normalization of the projective closure of $c(x,y)=0$. Let $\operatorname{div}(x)$ and $\operatorname{div}(y)$ be the divisors of $x$ and $y$ as functions of $C$. Recall that, given a function $h:C\to \mathbb{P}^1$, the divisor of $h$ is defined to be 
\[
\operatorname{div}(h)=\sum_{P\in C}\ord_P(h)P.
\]
Moreover, 
$\supp(\operatorname{div}(h))=\{P\in C: \ord_P(h)\neq 0\}$ is the set of poles and zeros of $h$.
We say that $\supp(\operatorname{div}(x))$ and $\supp(\operatorname{div}(y))$ are {\em comparable} if one is contained in the other.
In this section, we prove Theorem~\ref{thm:main} in the special cases where $c(x,y)$ is vertical or horizontal, or where $\supp(\operatorname{div}(x))$ and $\supp(\operatorname{div}(y))$ are comparable, or where $f$ or $c$ have large coefficients.

\begin{lemma}\label{lem:height-reduction}
Theorem~\ref{thm:main} holds if $c(x,y)$ is not absolutely irreducible, or if $H(c)$ or $H(f)$ exceeds $B^{O_D(1)}$, for a suitable implied constant.
\end{lemma}

\begin{proof}
Recall that $c(x,y)$ is irreducible. If it is not absolutely irreducible, then the upper bound in Theorem~\ref{thm:main} follows from the observation at the end of Section~\ref{s:goat}. Thus we may assume that $c(x,y)$ is absolutely irreducible.

By Lemma~\ref{lem:large}, if $H(c)$ is larger than $B^{O_D(1)}$, then $c(x,y)=0$ has $O_D(1)$ integer points of bounded height. For each $(x_0,y_0)$, there are at most $r$ choices of $z$ satisfying $f(x_0,y_0,z)=0$.

It remains to show that we can assume that $H(f)$ is bounded by a power of $B$. 
Consider the $\Q$-vector space $V$ spanned by \[
\{(x^p,y^q,z^r):(x,y,z)\in S(X,Y,Z,f)\}.
\] 
If $\rank(V)\leq 1$, then 
all vectors $(x^p,y^q,z^r)$ differ only by sign from a fixed integer vector on that line, 
since $\gcd(x,y,z)=1$. Each coordinate has $O_D(1)$ possible roots, which gives a total of $O_D(1)$ elements in $S(X,Y,Z,f)$. The rank of $V$ cannot be $3$, since each vector in $V$ is orthogonal to $(a_1,a_2,a_3)$. If $\rank(V)=2$, let $(x_1,y_1,z_1),(x_2,y_2,z_2)\in S(X,Y,Z,f)$ be such that $\mathbf{v}_1=(x_1^p,y_1^q,z_1^r)$ and $\mathbf{v}_2=(x_2^p,y_2^q,z_2^r)$ are linearly independent. But then $\mathbf{v}_1\times \mathbf{v}_2=m(a_1,a_2,a_3)$, for some $m\in \ZZ$, since $\gcd(a_1,a_2,a_3)=1$. Thus 
\[
H(f)\leq |\mathbf{v}_1\times \mathbf{v}_2|\leq B^{O_D(1)},
\] 
as required. 
\end{proof}

\begin{lemma}\label{lem:vertical-horizontal}
Theorem~\ref{thm:main} holds if $c(x,y)$ is vertical or horizontal.
\end{lemma}

\begin{proof}
Suppose first that $c(x,y)$ is vertical. Since $c(x,y)$ is absolutely irreducible, we must have $c(x,y)=\lambda(x-x_0)$, for some nonzero integer $\lambda$. We then have to count primitive solutions of $a_2y^q+a_3z^r=-a_1x_0^p$. If $x_0\neq 0$, the curve $a_2y^q+a_3z^r+a_1x_0^p=0$ is absolutely irreducible. Indeed, viewing it as a polynomial in $y$ over $\K(z)$, 
Lemma~\ref{lem:capelli} would force $a_3z^r+a_1x_0^p$ to be a proper power; this is impossible because $a_3z^r+a_1x_0^p$ has a nonzero constant term and simple roots. Bombieri--Pila, in the form proven by Heath-Brown in~\cite[Theorem 15]{HeathBrown}, gives $O_{\eps,D}(B^{\ve}\min\{Y^{1/r},Z^{1/q}\})$ points of bounded height. If $x_0=0$, on the other hand, we have $a_2y^q+a_3z^r=0$. Since $\gcd(x,y,z)=1$, $y$ and $z$ must be coprime and there are $O_D(1)$ solutions. The horizontal case is identical.
\end{proof}

From now on, we assume that $c(x,y)$ is neither vertical nor horizontal, so that $\operatorname{div}(x)$ and $\operatorname{div}(y)$ are both well-defined.

\begin{lemma}\label{lem:monomial}
Theorem~\ref{thm:main} holds if 
$\supp(\operatorname{div}(x))$ and $\supp(\operatorname{div}(y))$ are comparable. 
\end{lemma}

\begin{proof}
Assume that $\supp(\operatorname{div}(x))\subseteq\supp(\operatorname{div}(y))$.
	Let $\tilde{c}(y)=c(0,y)$, which is a nonzero polynomial since $c(x,y)$ is irreducible. Assume that there exists $y_0\in\overline{\Q}^*$ such that $\tilde{c}(y_0)=0$. Then 
	$P=(0,y_0)$ is in the support of $x$ but not of $y$, contradicting the assumption $\supp(\operatorname{div}(x))\subseteq\supp(\operatorname{div}(y))$. Thus it follows that 
	there is no such $y_0$, whence $\tilde{c}(y)=\delta y^e$ for $\delta\in \Q^*$ and $c(x,y)=xc_1(x,y)+\delta y^e$. Since $c\in \ZZ[x,y]$, it follows that in fact $\delta$ is a nonzero integer. 
	
	Let $(x_0,y_0,z_0)$ be a point counted by $S(X,Y,Z,f,c)$. 
	If $x_0=0$, then $\delta y_0^e=0$. If $e=0$, this is
impossible, while if $e>0$ it gives $y_0=0$, after which
$f=0$ forces $z_0=0$.
Hence we must have $x_0\neq 0$ and we put $g=\gcd(x_0,y_0)$. Note that $g\mid a_3$ since $\gcd(x_0,y_0,z_0)=1$. Let $x_0'=x_0/g$ and $y_0'=y_0/g$. Since $c(x_0,y_0)=0$, we have $gx_0'\mid \delta y_0'^eg^e$ and $x_0'\mid \delta y_0'^eg^{e-1}$, which implies $x_0'\mid \delta g^{e-1}$ since $\gcd(x_0',y_0')=1$. (If $e=0$, we have $x_0\mid \delta$.) 
	It follows that $x_0\mid \delta a_3^e$. Hence we have $O_{\ve,D}(B^{\ve})$ choices for $x_0$. For each such $x_0$, there are $O_D(1)$ choices for $y_0$, and then at most $r$ choices for $z_0$.
 The case $\supp(\operatorname{div}(y))\subseteq\supp(\operatorname{div}(x))$ is analogous.
\end{proof}

\subsection{Irreducibility}

Following our work in the previous section, we may assume that $c(x,y)$ is absolutely irreducible,
nonvertical and nonhorizontal, 
that $H(c),H(f)\leq B^{O_D(1)}$, 
and 
that $\supp(\operatorname{div}(x))$ and $ \supp(\operatorname{div}(y))$ are incomparable.
Moreover, we recall that $c$ has degree $d$.
We put $K=\overline{\Q}(C)$, on which the functions $x,y\in K$ are nonconstant.
Let 
\begin{equation}\label{eq:FF}
F(T)=a_3T^r+a_1x^p+a_2y^q\in K[T].
\end{equation} 
The main goal of this section is to determine when this polynomial is irreducible.
Given a function $h:C\to\mathbb{P}^1$, we define
\begin{equation}\label{eq:2deg}
 H_C(h)=\sum_{P\in C}\max\{0,-\ord_P(h)\}=\frac 12 \sum_{P\in C}\abs{\ord_P(h)}.
\end{equation}
In particular, we have $H_C(x), H_C(y)\leq d$.

\begin{lemma}\label{cor:Firr}
If $F(T)$ is reducible in $K[T]$, then
$\max\{p,q\}\leq4d^2$.
\end{lemma}

\begin{proof}
By Lemma~\ref{lem:capelli}, if $F(T)$ is reducible over $K[T]$, then there exists $\ell\geq2$ and $h\in K$ such that
\[
a_1x^p+a_2y^q+h^\ell=0
\]
in $K$.
We must have 
$a_1x^p+a_2y^q\neq0$ in $K$, since 
if $a_1x^p+a_2y^q$ vanishes identically on $C$, then $C$ would be an irreducible component of the projective closure of $a_1x^p+a_2y^q=0$, so that $\supp(\operatorname{div}(x))= \supp(\operatorname{div}(y))$. 
Put
\[
u=\frac{a_1x^p}{a_2y^q},
\qquad
v=\frac{h^\ell}{a_2y^q},
\qquad
H=H_C(u)=H_C(x^p/y^q).
\]
Then $1+u+v=0$, with no vanishing proper subsum. Moreover,
$v=-(1+u)$ has the same pole divisor as $u$.

Let $S$ be the set of zeros and poles of $u$ and $v$. The
poles of $h$ are contained among the poles of $x$ and $y$.
The zeros of $x$ and $y$ contribute at most
$H_C(x)+H_C(y)\leq2d$ points, while their poles lie above the
line at infinity and contribute at most $d$ points. Every
remaining zero of $h$ is a zero of $v$ of multiplicity at
least $\ell$, and hence there are at most $H/\ell$ such
points. It follows that 
\[
|S|\leq\frac H\ell+3d.
\]
The Brownawell--Masser theorem, as recorded in 
Lemma~\ref{lem:Sunit}, 
gives
\[
H\leq2g(C)-2+|S|
 \leq (d-1)(d-2)-2+\frac H\ell+3d
 =d^2+\frac H\ell.
\]
Consequently,
\[
H\leq\frac{\ell}{\ell-1}d^2\leq2d^2.
\]

Since the supports of $\div(x)$ and $\div(y)$ are
incomparable, there is a point $P\in C$ with
$\ord_P(x)\neq0$ and $\ord_P(y)=0$. Hence, by~\eqref{eq:2deg},
\[
p\leq|\ord_P(u)|\leq2H\leq4d^2.
\]
Interchanging $x$ and $y$ gives $q\leq4d^2$.
\end{proof}

We now show that the assumption $\max\{p,q\}\leq 4d^2$ is not necessary when the curve $c(x,y)=0$ is a line or a conic.

\begin{lemma}\label{lem:lineirr}
Suppose that $d=1$.
Then $F(T)$ is irreducible in $K[T]$.
\end{lemma}

\begin{proof}
	Since $C$ is not vertical, $x$ gives a rational parameter on $C$. Since $C$ is not horizontal and $\supp(\operatorname{div}(x))\neq \supp(\operatorname{div}(y))$, after setting $t=x$ we may write $y=ut+v$ with $u,v\in\K^*$. Thus $K=\K(t)$, and it is enough to prove that $a_3T^r+a_1t^p+a_2(ut+v)^q$ is irreducible in $\K(t)[T]$. By Lemma~\ref{lem:capelli}, if $F(T)$ is reducible, then there exists $h\in\K(t)$ such that $a_1t^p+a_2(ut+v)^q= h(t)^\ell$, with $\ell\geq2$. In fact $h$ is a polynomial; indeed, 
	if $h\in \K(t)$ had a finite pole, so would $h^\ell$, whereas $a_1t^p+a_2(ut+v)^q$ is a polynomial.
	
	We now apply Lemma~\ref{lem:abc} to $a_1t^p+a_2(ut+v)^q- h(t)^\ell=0$. The three polynomials are pairwise coprime. (The factors $t$ and $ut+v$ are coprime because $v\neq0$.)
	Thus the Mason--Stothers theorem gives 
	\[\max\{p,q,\ell\deg h\}\leq \deg\rad(t^p(ut+v)^qh^\ell)-1.
	\] 
	Since $\deg\rad(t^p(ut+v)^qh^\ell)\leq 2+\deg h$, we obtain $\max\{p,q,\ell\deg h\}\leq \deg h+1$. Because $\ell\geq2$, this implies $\deg h\leq1$. Since $p,q\geq2$, we must have $p=q=2$, $\ell=2$, and $\deg h=1$.
	It remains to rule out this last possibility. If $p=q=2$, then $a_1t^2+a_2(ut+v)^2$ must be a square in $\K(t)$ for $F(T)$ to be reducible. Its discriminant is $-4a_1a_2v^2$, which is nonzero, so $a_1t^2+a_2(ut+v)^2$ cannot be a square. It now follows from Lemma~\ref{lem:capelli} that $F(T)$ is irreducible in $K[T]$.
\end{proof}

Assume now $d=2$. We say that $c(x,y)=0$ is an ellipse, hyperbola, or parabola according to whether the discriminant of the quadratic part of $c(x,y)$ is negative, positive, or zero, respectively. 

\begin{lemma}\label{lem:nopar}
Suppose that $d=2$ and $c(x,y)=0$ is not a parabola. Then Theorem~\ref{thm:main} holds.
\end{lemma}
\begin{proof}
In this case there exist nonzero integers $A,a,b,e$ of size at most $B^{O_D(1)}$, and 
affine linear polynomials $U,V\in \ZZ[x,y]$ 
with coefficients of size at most $B^{O_D(1)}$
and whose linear parts are linearly independent,
such that 
 \[
A c(x,y)=aU(x,y)^2+bV(x,y)^2-e.
 \]
Moreover the map 
$(x,y)\mapsto (U(x,y),V(x,y))$ is injective, with an image that lies in a box of side length 
$B^{O_D(1)}$.
 We may now apply Lemma~\ref{lem:VW} to deduce that 
 $\#S(X,Y,Z,f,c)=O_{\ve,D}(B^\ve)$.
 \end{proof}

We may proceed under the assumption that $c(x,y)=0$ defines a parabola. In this case, points on $c(x,y)=0$ can be parametrized by $(x,y)=(\gamma_1(t),\gamma_2(t))$ with $\gamma_1,\gamma_2\in\K[t]$ such that $\max\{\deg(\gamma_1),\deg(\gamma_2)\}=2$. 
In this case we have $K=\K(C)=\K(t)$.

\begin{lemma}\label{lem:c00}
Suppose that $d=2$, that $c(x,y)=0$ is a parabola and that $c(0,0)=0$. Then Theorem~\ref{thm:main} holds.
\end{lemma}
\begin{proof}
	Let $c(x,y)=\alpha x^2+\beta xy+\gamma y^2+\Delta x+\eta y+\upsilon$. We have $\upsilon=0$ and $\beta^2=4\alpha \gamma$.
	
	Since the quadratic part has rank one, there exist a nonzero
integer $\delta$ and a primitive integral linear form $L$ such
that
$
\alpha x^2+\beta xy+\gamma y^2=\delta L(x,y)^2,
$
with $|\delta|=\gcd(\alpha,\beta,\gamma)$.	
Then $c(x,y)=\delta L(x,y)^2+M(x,y)$, where $L(x,y)=\ell_1 x+\ell_2 y$ and $M(x,y)=m_1 x+m_2 y$ are linear forms with integer coefficients. Since $c(x,y)$ is absolutely irreducible, $L$ and $M$ are linearly independent. Set $t=L(x,y)$. Then, for points on the curve, we have $M(x,y)=-\delta t^2$. Hence
 \[
 x=\frac{m_2 t+\ell_2\delta t^2}{\ell_1m_2-m_1\ell_2} \hspace{20 pt} \text{and} \hspace{20 pt} y=\frac{-m_1 t-\ell_1\delta t^2}{\ell_1m_2-m_1\ell_2}.
 \]
 If $(x,y)\neq (0,0)$ is an integral point on $c(x,y)=0$, then $t=L(x,y)\in \Z_{\neq 0}$.
	Putting $D_0=\ell_1m_2-m_1\ell_2$ and $(x,y)=(gu,gv)$, where $\gcd(u,v)=1$, we 
may write $t=gs$, for some integer $s$. But then 
\[
D_0u=s(m_2+\ell_2\delta gs) \hspace{20 pt} \text{and} \hspace{20 pt}
D_0v=-s(m_1+\ell_1\delta gs).
\]
This implies that $s\mid D_0$, whence $t\mid D_0g$. It follows that $t\mid a_3D_0$, since $g\mid a_3$ for $(x,y,z)\in S(X,Y,Z,f,c)$.
Thus there are $O_{\ve,D}(B^{\ve})$ choices for $t$. Once $t$ is fixed, $x$ and $y$ are determined and there are then $O_D(1)$ choices for $z$.
\end{proof}

\begin{lemma}\label{lem:irrsquare}
 Let $\gamma_1,\gamma_2\in\K[t]$ be such that $\max\{\deg(\gamma_1(t)),\deg(\gamma_2(t))\}=2$. Assume that $\gcd(\gamma_1(t),\gamma_2(t))=1$ and that there exists nonzero $h\in \K[t]$ such that \[
 a_1\gamma_1(t)^p+a_2\gamma_2(t)^q=h(t)^m. 
 \]
 Then $m\leq 2$.
\end{lemma}

\begin{proof} 
Write $d_i=\deg(\gamma_i(t))$ and $A=\max\{pd_1,qd_2\}$. 
The three polynomials $\gamma_1^p,\gamma_2^q,h^m$ are pairwise coprime. Hence 
Lemma~\ref{lem:abc}
gives \[ A \leq d_1+d_2+\deg h-1 \leq d_1+d_2+\frac Am-1. \] Since $\max\{d_1,d_2\}=2$ and $p,q\geq2$, we have $A\geq4$. It follows that \[ A\leq\frac{3m}{m-1}, \] 
whence $m\leq4$. Suppose that $m\geq3$. 
Then 
\[ A\leq\frac32(d_1+d_2-1). 
\] 
If $d_1+d_2\leq3$, the right-hand side is at most $3$, contrary to $A\geq4$. Thus $d_1=d_2=2$, and the same inequality gives $\max\{p,q\}\leq9/4$. Therefore $p=q=2$
and $\deg h\leq1$. 
Over $\overline{\QQ}$, put $P_1=\sqrt{a_1}\,\gamma_1$ and $P_2=\sqrt{-a_2}\,\gamma_2$. 
Then \[ (P_1-P_2)(P_1+P_2)=h^m, \] and the two factors on the left are coprime. If $h$ is constant, then $P_1$ and $P_2$ are constant, a contradiction. If $\deg h=1$, coprimality implies that one of $P_1-P_2$ and $P_1+P_2$ is constant and the other is a constant multiple of $h^m$. It follows that $\deg P_1=m$, which contradicts the fact that $\deg P_1=2$. Hence $m\leq2$. \end{proof}

\begin{lemma}\label{lem:conicirr}
Suppose that $d=2$, that $c(x,y)=0$ is a parabola and that $c(0,0)\neq 0$. 
If 
$F(T)$ is reducible in $K[T]$, then $r$ is even and $F(T)=a_3(T^{r/2}-G(x,y))(T^{r/2}+G(x,y))$ is the product of two irreducible polynomials of degree $r/2$ in $K[T]$.
\end{lemma}
\begin{proof}
 By the parametrization of parabolas, points on $c(x,y)=0$ can be parametrized by $(x,y)=(\gamma_1(t),\gamma_2(t))$ with $\gamma_1,\gamma_2\in \K[t]$ such that 
 $\max\{\deg(\gamma_1),\deg(\gamma_2)\}=2$. Since $c(0,0)\neq 0$, we must have $\gcd(\gamma_1(t),\gamma_2(t))=1$. Notice that $K=\K(C)=\K(t)$. 
 Using the same completion of squares parametrization as in 
the proof of Lemma~\ref{lem:c00}, we may choose the parameter so that $t=L(x,y)$ for a linear form $L$.
Let 
	\[\alpha(t)=-\frac{a_1\gamma_1(t)^p+a_2\gamma_2(t)^q}{a_3}\in K^*.
	\]
Suppose that $F(T)$ is reducible in $K[T]$. Then $\alpha(t)$ must be an $\ell$-th power for $\ell\geq 2$ in $K$, by Lemma~\ref{lem:capelli}, with $\ell\mid r$.
Moreover, its $\ell$-th root must be a polynomial since $\alpha(t)$ is a polynomial and has no finite poles. It follows from Lemma~\ref{lem:irrsquare} that $\ell=2$.
Thus $a_3^{-1}F(T)=(T^{r/2}-h(t))(T^{r/2}+h(t))$ for $h(t)^2=\alpha(t)$. 
	Each factor is irreducible. Indeed, if $T^{r/2}\pm h(t)$ were reducible, 
	then Lemma~\ref{lem:capelli} would imply that $\pm h(t)$ is an $m$-th power in $K$ for some $m\geq2$. But then we would have that $\alpha(t)=h(t)^2$ is a $(2m)$-th power in $K$, contradicting Lemma~\ref{lem:irrsquare}.
We conclude by putting $G(x,y)=h(L(x,y))$.
\end{proof}

\subsection{Completion of the proof of Theorem~\ref{thm:main}}

Recall that $B=\max\{X,Y,Z\}$, that $d\leq D$ is the degree of $c(x,y)$ and that 
 $2\leq p,q,r\leq D$.
As in the previous section, we may assume that $c(x,y)$ is absolutely irreducible,
nonvertical and nonhorizontal, 
that $H(c),H(f)\leq B^{O_D(1)}$, 
and 
that $\supp(\operatorname{div}(x))$ and $ \supp(\operatorname{div}(y))$ are incomparable.
We put $K=\overline{\Q}(C)$, on which the functions $x,y\in K$ are nonconstant, and we recall the definition~\eqref{eq:FF} of $F\in K[T]$.

\begin{lemma}\label{lem:irreducible-case}
	Let $\ve>0$. Assume that $F(T)$ is irreducible in $K[T]$. Then
	\[
	\#S(X,Y,Z,f,c)\ll_{\eps,D}B^{\ve}\max\{X,Y\}^{\frac{1}{dr}}.
	\]
\end{lemma}

\begin{proof}
Let 
\[
A=\K[x,y]/(c),\qquad K=\operatorname{Frac}(A).
\]
Since $c$ is absolutely irreducible, $A$ is an integral domain. Moreover,
$a_3^{-1}F(T)$ is monic and irreducible in $K[T]$. It follows that
$A[z]/(F)$ is an integral domain, and hence that
$L=\operatorname{Frac}\bigl(A[z]/(F)\bigr)$
is well-defined and has degree
$[L:K]=r$ over $K$.
The element $z\in L$ is nonconstant, since otherwise $F(T)$ would have
a root in $K$, contrary to its irreducibility. Thus $L$ is a finite
extension of $\K(z)$, with 
$[L:\K(z)]=O_D(1)$,
by B\'ezout's theorem applied to the curve
$c(x,y)=f(x,y,z)=0$.

We next choose a suitable linear projection. Since
$L=\K(z)(x,y)$, the primitive element theorem shows that
$L=\K(z)(x+\lambda y)$
for all but $O_D(1)$ values of $\lambda\in\K$. Indeed, if
$\sigma_i\neq\sigma_j$ are two $\K(z)$-embeddings of $L$ into a
normal closure, then the condition
\[
\sigma_i(x+\lambda y)=\sigma_j(x+\lambda y)
\]
excludes at most one value of $\lambda$, and there are only $O_D(1)$
pairs of embeddings.

Let $c_d(x,y)$ denote the homogeneous degree $d$ part of $c$.
If
$c_d(-\lambda,1)\neq0$,
then the polynomial
\[
c_\lambda(w,y)=c(w-\lambda y,y)
\]
has degree exactly $d$ in $y$. Since the change of variables
$(x,y)\mapsto(w,y)=(x+\lambda y,y)$ is invertible, the polynomial
$c_\lambda(w,y)$ is irreducible in $\K[w,y]$. Consequently,
we have $[K:\K(x+\lambda y)]=d$.
After excluding a further $O_D(1)$ values, we may therefore choose
an integer $\lambda$, with $|\lambda|\ll_D1$, such that, on writing
$w=x+\lambda y$,
we have
\[
L=\K(z,w)
\qquad\text{and}\qquad
[K:\K(w)]=d.
\]
It follows that
$
[L:\K(w)]
=
[L:K][K:\K(w)]
=
rd $.

Let $P(T)\in\K(w)[T]$ be the minimal polynomial of $z$ over
$\K(w)$. Thus $\deg P=rd$. After clearing denominators and taking
the primitive part, we obtain an absolutely irreducible polynomial
$h\in\K[W,T]$
such that
$h(w,z)=0$, with $\deg_T h=rd$.
We claim that $\deg h=O_D(1)$. To see this, consider
\[
R(W,T)
=
\Res_Y\bigl(
c(W-\lambda Y,Y),
f(W-\lambda Y,Y,T)
\bigr).
\]
The two polynomials inside the resultant are coprime in
$\K(W,T)[Y]$, since otherwise the irreducible polynomial
$c(W-\lambda Y,Y)$ would divide
$f(W-\lambda Y,Y,T)$, which is impossible since the coefficient
of $T^r$ in the latter polynomial is the nonzero constant $a_3$.
Hence $R$ is nonzero. Moreover,
$\deg R=O_D(1)$ and $H(R)\leq B^{O_D(1)}$, since $|\lambda|\ll_D1$ and $H(c),H(f)\leq B^{O_D(1)}$.
Since $y$ is a common zero of
$c(w-\lambda Y,Y)$ and $f(w-\lambda Y,Y,z)$
in $L$, we have $R(w,z)=0$. The minimal polynomial $P(T)$ therefore
divides $R(w,T)$ in $\K(w)[T]$, and Gauss' lemma shows that
$h(W,T)$ divides $R(W,T)$ in $\K[W,T]$, whence 
$\deg h=O_D(1)$.

 If the curve 
 $h(W,T)=0$ is not defined over $\Q$, there are $O_{D}(1)$ integral points on it. Thus we may assume that $h\in\Z[W,T]$ is primitive and absolutely irreducible. 
 Moreover, since $h$
is a rational factor of $R$, the preceding height bound also gives
$H(h)\leq B^{O_D(1)}$.
Every point $(x,y,z)$ counted by $S(X,Y,Z,f,c)$ gives an integral
point $(w,z)$ on $h=0$ satisfying
\[
|w|\leq (1+|\lambda|)\max\{X,Y\}\ll_D \max\{X,Y\},
\qquad
|z|\leq Z.
\]
Conversely, for fixed $(w,z)$, the equation
$c(w-\lambda y,y)=0$
has at most $d$ solutions in $y$, since it has degree exactly $d$
in $y$. It therefore remains to count the relevant integral points
on $h(W,T)=0$.
For this we appeal to Heath-Brown’s lopsided curve estimate~\cite[Theorem~15]{HeathBrown}, which gives 
 \[
\ll_{\ve,D} B^{\ve}\exp\left(\frac{\log \max\{X,Y\}\log Z}{dr\log Z}\right)
\ll_{\ve,D}B^{\ve}\max\{X,Y\}^{\frac{1}{dr}}
 \]
relevant points.
\end{proof}

\begin{lemma}\label{lem:kummer-case}
	Let $\ve>0$. Assume that $F(T)$ is reducible in $K[T]$. Then
\[\#S(X,Y,Z,f,c)\ll_{\eps,D}\max\{X,Y\}^{\frac{1}{m(p,q)}+\eps},
\]
where $m(p,q)$ is defined in~\eqref{eq:mpq}.
\end{lemma}

\begin{proof}
It follows from Lemma~\ref{cor:Firr} that $d\geq m(p,q)$. The bound now follows from applying Lemma~\ref{lem:BP} to $c(x,y)=0$.
\end{proof}

\begin{proof}[Proof of Theorem~\ref{thm:main}]
We first prove~\eqref{eq:main-bound}. If $F$ is irreducible over $K[T]$ then 
the desired bound follows from Lemma~\ref{lem:irreducible-case}. 
Alternatively, if $F$ is reducible then we can apply Lemma~\ref{lem:kummer-case}.

It remains to prove~\eqref{eq:main-bound2}. When $d=1$ the desired bound follows from combining 
 Lemmas~\ref{lem:lineirr} and~\ref{lem:irreducible-case}.
 Suppose next that $d=2$. We can assume that $c(x,y)=0$ is a parabola
 with $c(0,0)\neq0$, by Lemmas~\ref{lem:nopar} and~\ref{lem:c00}. It follows from Lemma~\ref{lem:conicirr} that $F(T)$ is either irreducible or the product of two irreducible factors of degree $r/2$. In the first case, we may apply Lemma~\ref{lem:irreducible-case}. In the second case, we prove Lemma~\ref{lem:irreducible-case} for $T^{r/2}\pm G(x,y)$ instead of $a_3T^r+a_1x^p+a_2y^q$ and get the bound
 \[
 \#S(X,Y,Z,f,c)\ll_{\eps,D}B^{\ve}\max\{X,Y\}^{\frac{2}{2r}},
 \]
 which is satisfactory.
\end{proof}

\section{Counting points on the generalized Fermat surface}

We now have everything in place to prove Theorem~\ref{thm:main2}. 
Let $X,Y,Z\geq 2$ and put $B=\max\{X,Y,Z\}$. Set $D=\max\{p,q,r\}$.
Recall the definition 
\eqref{eq:W} of $W$. We shall apply the determinant method.
According to~\cite[Theorem~1.1]{bv}, 
which is based on work of Salberger~\cite{saldet}, there exist polynomials
$f_1,\ldots, f_J\in \Z[x,y,z]$, 
and a finite collection $Q$ of points such that 
the following
hold:
\begin{enumerate}
	\item $J\ll_{\ve,D} B^\ve W $.
	\item
	Each $f_j$ is coprime to $f$ and has degree 
	$O_{\ve,D}(\log B)$, for $j\leq J$. 
	\item
	$\#Q\ll_{\ve,D} B^\ve W^{2}$.
	\item
	For each $(x,y,z)\in S(X,Y,Z,f)\setminus Q$, we have
	$f_j(x,y,z)=0$,
	for some $j\leq J$.
\end{enumerate}
Recalling the definition~\eqref{eq:hen} of 
$S(X,Y,Z,f)$, 
it therefore follows that 
\[
\#S(X,Y,Z,f)\ll_{\ve,D} B^\ve W^2+\sum_{j\leq J} N_j(X,Y,Z),
\]
where 
$B=\max\{X,Y,Z\}$ and 
$N_j(X,Y,Z)$ is the number of $(x,y,z)\in S(X,Y,Z,f)$ such that $f_j(x,y,z)=0$.

The variety cut out by $f(x,y,z)=f_j(x,y,z)=0$ is a curve in $\mathbb{A}^3$, with $O_{\ve,D}(\log B)$ irreducible components, each of degree $O_{\ve,D}(\log B)$.
For each auxiliary polynomial $f_j$, define 
\[ 
R_j(x,y)= \begin{cases} \Res_z(f,f_j) & \text{if $\deg_z f_j>0$},\\ 
f_j(x,y) & \text{if $\deg_z f_j=0$}. 
\end{cases} 
\] 
Since $f$ is absolutely irreducible and $f_j$ is coprime to $f$, the polynomial $R_j$ is nonzero. Factor its primitive part over $\QQ$ as 
\[ 
R_j(x,y)=\prod_{m=1}^{s_j}c_{j,m}(x,y)^{e_{j,m}}, 
\] 
where the $c_{j,m}$ are pairwise non-associate and irreducible over $\QQ$. 
Every point counted by $N_j(X,Y,Z)$ satisfies $R_j(x,y)=0$, and therefore 
\[ 
N_j(X,Y,Z) \leq \sum_{m=1}^{s_j} \#S(X,Y,Z,f,c_{j,m}),
\] 
in the notation of~\eqref{eq:duck}.
The standard degree bound for resultants gives 
\[ 
\sum_{m=1}^{s_j}\deg c_{j,m} \ll_D \deg f_j \ll_{\varepsilon,D}\log B. 
\] 
Thus both the number and the degrees of the factors are $O_{\varepsilon,D}(\log B)$. 
If $\deg c_{j,m}\geq D$
 then~\cite[Theorem~2]{bincluck} gives 
\[ 
\#S(X,Y,Z,f,c_{j,m})
\ll_\ve (\deg c_{j,m})^2 \max\{X,Y\}^{\frac{1}{\deg c_{j,m}}+\ve} \ll_{\ve,D}
B^\varepsilon\max\{X,Y\}^{\frac{1}{r}}, 
\] 
since $r\leq D\leq \deg c_{j,m}\ll_{\ve,D} \log B$. 
If $\deg c_{j,m}\in \{1,2\}$, 
we apply~\eqref{eq:main-bound2} in Theorem~\ref{thm:main} to get
\[
\#S(X,Y,Z,f,c_{j,m})\ll_{\ve,D}B^{\ve}\max\{X,Y\}^{\frac{1}{r}},
\]
which is satisfactory. Finally, we may suppose that 
 $3\leq \deg c_{j,m}< D$. But then it follows from 
combining~\eqref{eq:main-bound} with 
 Lemma~\ref{lem:BP} that 
 \[
\#S(X,Y,Z,f,c_{j,m})\ll_{\ve,D} B^{\ve}\left(\max\{X,Y\}^{\min\left\{\frac{1}{3},\frac{1}{m(p,q)}\right\}}+\max\{X,Y\}^{\frac{1}{r}}\right),
\]
where $m(p,q)$ is given by~\eqref{eq:mpq}.
Clearly $\min\{\frac{1}{3},\frac{1}{m(p,q)}\}=\frac{2}{\sqrt{\max\{p,q,36\}}}$. 
Finally, on summing over $m$ and $j$, and absorbing powers of $\log B$ into 
$B^\varepsilon$, we obtain the bound 
\[ 
\sum_{j\leq J}N_j(X,Y,Z) \ll_{\varepsilon,D} B^\varepsilon W \left( \max\{X,Y\}^{\frac{2}{
\sqrt{\max\{p,q,36\}}}}
+ \max\{X,Y\}^{\frac{1}{r}} \right). 
\] 
This finally concludes the proof of Theorem~\ref{thm:main2}.

\begin{remark}\label{rem:rem}
The reducible case admits a small refinement. Put
\[
\alpha=-\frac{a_1x^p+a_2y^q}{a_3},
\]
and let $k\geq2$ be the largest divisor of $r$ such that
$\alpha=h^k$ for some $h\in K$. By maximality of $k$ and
Lemma~\ref{lem:capelli}, each factor in
\[
a_3^{-1}F(T)
 =T^r-h^k
 =\prod_{\zeta^k=1}\bigl(T^{r/k}-\zeta h\bigr)
\]
is irreducible over $K$. Repeating the proof of
Lemma~\ref{lem:irreducible-case} on each component therefore gives
\[
\#S(X,Y,Z,f,c)
 \ll_{\ve,D}
 B^\ve\max\{X,Y\}^{\frac{k}{rd}}.
\]
The proof of Lemma~\ref{cor:Firr} also gives
$\max\{p,q\}\leq 2kd^2/(k-1)$,
whence
\[
\frac{k}{rd}
 \leq
 \sqrt{\frac{2k^3}{r^2(k-1)\max\{p,q\}}}
 \leq
 \sqrt{\frac{2r}{(r-1)\max\{p,q\}}}.
\]
Thus
\[
\#S(X,Y,Z,f,c)
 \ll_{\ve,D}
 B^\ve\max\{X,Y\}^{\frac{1}{m_r(p,q)}},
 \qquad
 m_r(p,q)=
 \sqrt{\frac{(r-1)\max\{p,q\}}{2r}}.
\]
This recovers Lemma~\ref{lem:kummer-case} when $r=2$, but improves it for  $r\geq3$. Thus, the exponent $1/m(p,q)+\ve$ can be improved to $1/m_r(p,q)+\ve$ in~\eqref{eq:main-bound}, and $\max\{p,q,36\}$ in Theorem~\ref{thm:main2} can be replaced by 
\[
\max\left\{\frac{2(r-1)p}{r},\frac{2(r-1)q}{r},36\right\}.
\]
\end{remark}

\section{Sums of three powerful numbers}

Our goal in this section is to prove Theorem~\ref{t:uv-main} using Theorem~\ref{thm:main2}. We shall begin by recording a general transference principle, which shows precisely how 
uniform upper bounds for points on the generalized Fermat surface feed into non-trivial upper bounds for 
the counting function $N(B)$ defined in~\eqref{eq:hawk}.

\begin{theorem}\label{prop:etagen}
Let $p\geq q\geq r\ge 2$. Assume that there is a fixed $\kappa\geq 0$ such that, uniformly for nonzero $\lambda,\mu,\nu\in\mathbb Z$, for every $0\le k\le r-1$, for all $X,Y,Z\geq 2$ and $\ve>0$, we have 
\[
\#\left\{(x,y,z)\in \mathbb N^3:
\begin{array}{l}
x\le X,\ y\le Y,\ z\le Z,\ \gcd(x,y,z)=1,\\
\lambda x^p+\mu y^q=\nu z^{r+k}
\end{array}
\right\}
\ll_{\varepsilon,p}
H^{\ve}\max\{X,Y\}^{\kappa}
\]
with $H=\max\{X,Y,Z\}$.
Then 
\[
N(B)\ll_{\delta, p}B^{\frac{1}{p}+\frac{1}{q}-\delta},
\]
for any 
\[
\delta<
 \frac{
 \displaystyle
 \frac1p+\frac1q-
 \left(
 \frac1r-\frac1{qr}+\frac{\kappa}q
 \right)}
 {p+q+1+1/r}.
\] 
\end{theorem}

\begin{proof}
Any nonzero $m$-full integer $n\in \mathbb{N}$ can be written uniquely in the form
\[
n = v_0^{m} \prod_{s = 1}^{m-1} v_s^{m + s},
\]
for $v_0, v_1, \ldots, v_{m - 1} \in \mathbb{N}$, such that $\mu^2(v_s) =1$ for $1 \leq s \leq m-1$ and $\gcd(v_s, v_{s'}) = 1$ for $1 \leq s < s' \leq m-1$.
Recall 
from~\eqref{eq:hawk} that 
\[
N(B)=\#\{(a,b,c)\in (\mathcal{S}_p\times \mathcal{S}_q\times \mathcal{S}_r)\cap[1,B]^3: \gcd(a,b,c)=1,~a+b=c\}.
\]
We adopt the factorisation
\[
a=x_0^p\prod_{i=1}^{p-1}x_i^{p+i},\qquad
b=y_0^q\prod_{j=1}^{q-1}y_j^{q+j},\qquad
c=z_0^r\prod_{k=1}^{r-1}z_k^{r+k}.
\]
The squarefree and coprimality conditions in this parametrisation make it unique;
for upper bounds, they can only reduce the number of admissible tuples. 

We decompose all variables dyadically. On a fixed dyadic box, write
\[
x_i\sim B^{\alpha_i},\qquad y_j\sim B^{\beta_j},\qquad z_k\sim B^{\gamma_k}.
\]
Let
\begin{equation}\label{eq:UVW}
S=\sum_{i=0}^{p-1}\alpha_i,\qquad
T=\sum_{j=0}^{q-1}\beta_j,\qquad
U=\sum_{k=0}^{r-1}\gamma_k.
\end{equation}
The conditions $a,b,c\le B$ imply
\[
\sum_{i=0}^{p-1}(p+i)\alpha_i\le 1,\qquad
\sum_{j=0}^{q-1}(q+j)\beta_j\le 1,\qquad
\sum_{k=0}^{r-1}(r+k)\gamma_k\le 1.
\]
In particular, we have 
\[
S\le \frac1p,\qquad T\le \frac1q,\qquad U\le \frac1r.
\]

Fix a dyadic box, and suppose that its contribution is
$O_{p}(B^{\Delta})$ on this box. 
The argument behind the trivial bound~\eqref{eq:triv} easily gives
\begin{equation}\label{eq:triv'}
\Delta\le S+T,\qquad \Delta\le S+U,\qquad \Delta\le T+U.
\end{equation}
The fourth bound we will use comes from fixing every variable except $x_0,y_0,z_k$, for a choice of $0\le k\le r-1$. The equation becomes
\[
\lambda x_0^p+\mu y_0^q=\nu z_k^{r+k}
\]
with fixed nonzero coefficients $\lambda,\mu,\nu$. Since $\gcd(a,b,c)=1$, we have
$\gcd(x_0,y_0,z_k)=1$. Therefore, for every $0\le k\le r-1$, it follows from the hypothesis of the theorem that 
\begin{equation}\label{eq:4way}
\Delta
\le
S+T+U-\alpha_0-\beta_0-\gamma_k
+\kappa\max\{\alpha_0,\beta_0\}+\ve.
\end{equation}

Suppose that
\[
\Delta> \frac 1p+\frac 1q-\eta,
\]
where 
\[
\eta=
 \frac{
 \displaystyle
 \frac1p+\frac1q-
 \left(
 \frac1r-\frac1{qr}+\frac{\kappa}q
 \right)}
 {p+q+1+1/r}\in \R. 
\]
Since $\Delta\le S+T$, by~\eqref{eq:triv'}, 
while $S\le 1/p$ and $T\le 1/q$, it follows that 
\[
S\ge \frac1p-\eta,\qquad T\ge \frac1q-\eta.
\]
Moreover, since $S=\alpha_0+\sum_{i=1}^{p-1}\alpha_i$, we have
$
1\ge pS+(S-\alpha_0).
$
Hence
\begin{equation}\label{eq:Ua0}
 S-\alpha_0\le 1-pS\le p\eta.
\end{equation}
Similarly,
\begin{equation}\label{eq:Vb0}
T-\beta_0\le q\eta.
\end{equation}

The bound $\Delta\le S+U$ 
in~\eqref{eq:triv'}
also gives
$
U\ge \Delta-S\ge \frac1q-\eta.
$
Let
$M=\max_{0\le k\le r-1}\gamma_k$,
and choose $k$ such that $\gamma_k=M$. Since there are $r$ variables on the $z$-side, we obtain
\begin{equation}\label{eq:MWr}
 M\ge \frac Ur\ge \frac1{qr}-\frac{\eta}{r}.
\end{equation}
Combining~\eqref{eq:4way},~\eqref{eq:Ua0} and~\eqref{eq:Vb0}, we obtain
\[
\Delta
\le
(p+q)\eta+(U-M)
+\kappa\max\{\alpha_0,\beta_0\}+\ve.
\]
Since $p\ge q$, we have 
$
\max\{\alpha_0,\beta_0\}
\le
\frac 1q.
$
Substituting this and~\eqref{eq:MWr} into the bound for $\Delta$, we obtain
\begin{equation}\label{eq:finaleta}
\Delta
\le
(p+q)\eta+
\frac1r-\frac1{qr}+\frac{\eta}{r}+\frac{\kappa}{q}+\ve=
\frac 1p+\frac 1q-\eta+\ve.
\end{equation}
Recall that we were assuming that $\Delta> \frac 1p+\frac 1q-\eta $. But if not, then~\eqref{eq:finaleta} clearly holds.
Thus, no dyadic box can contribute with exponent more than $\frac1p+\frac1q-\eta+\ve$.
The number of dyadic boxes is $O_{p}((\log B)^{p+q+r})$, which is
absorbed into $B^\varepsilon$. Hence
\[
N(B)\ll_{\varepsilon,p}B^{\frac 1p+\frac 1q-\eta+\varepsilon},
\]
for any $\ve>0$, which thereby completes the proof. 
\end{proof}

\begin{remark}
 Theorem~\ref{prop:etagen} beats the trivial bound~\eqref{eq:triv} precisely when we can take $\delta>0$, or
equivalently when
\[
\frac1p+\frac1q-\frac1r
>
\frac1q\left(\kappa-\frac1r\right).
\]
When the required estimate is supplied by Theorem~\ref{thm:main2}, one may take
\[
\kappa
=
\max\left\{
\frac{2}{\sqrt r},
\frac1{\sqrt r}
+\min\left\{\frac2{\sqrt p},\frac13\right\}
\right\}
\leq \frac3{\sqrt r}.
\]
Since $q\geq r$, the right-hand side in the preceding criterion is
$O(r^{-3/2})$. Consequently, for every fixed $\omega>0$, the
condition
\[
\frac1p+\frac1q>\frac{1+\omega}{r}
\]
ensures a power saving as soon as $r\gg\omega^{-2}$. 
In Corollary~\ref{cor:general-campana-bound} we will see how a 
refined optimisation 
yields savings in a larger
range, including some triples with
$1/p+1/q\leq1/r$.
 \end{remark}

We now deduce some consequences of Theorem~\ref{prop:etagen}. Before doing so, we need the following technical lemma.

\begin{lemma}\label{lem:gammarho}
 Let $r\geq2$ and $\gamma_0,\dots,\gamma_{r-1}\geq0$ be such that $\sum_{i=0}^{r-1}(r+i)\gamma_i\leq1$. Let $\rho\geq 0$ and $\gamma=\max_{0\leq i\leq r-1}\gamma_i$. Then
 \[
 \sum_{i=0}^{r-1}\gamma_i -\gamma(1-\rho)\leq \max_{0\leq j\leq r-1}\frac{2(j+\rho)}{(j+1)(2r+j)}.
 \]
\end{lemma}

\begin{proof}
By the rearrangement inequality, replacing the sequence
$(\gamma_i)$ by its decreasing rearrangement cannot increase the weighted sum, 
while leaving $\sum_{i=0}^{r-1}\gamma_i$ and $\gamma$ unchanged.
Put $\gamma_r=0$ and $\delta_j=\gamma_j-\gamma_{j+1}$ for $0\leq j\leq r-1$. Then $\delta_j\geq0$. If $A_j=\sum_{i=0}^j(r+i)=(j+1)(2r+j)/2$, then $\sum_{i=0}^{r-1}(r+i)\gamma_i=\sum_{j=0}^{r-1}A_j\delta_j$, so that 
\[
\sum_{i=0}^{r-1}\gamma_i-\gamma(1-\rho)=\sum_{j=0}^{r-1}(j+\rho)\delta_j\leq \left(\max_{0\leq j\leq r-1}\frac{(j+\rho)}{A_j}\right)\sum_{j=0}^{r-1}A_j\delta_j\leq \max_{0\leq j\leq r-1}\frac{(j+\rho)}{A_j},
\]
which completes the proof.
\end{proof}

For $\theta\geq0$, define
\[
\Lambda_r(\theta)
 =
 \max_{0\leq j\leq r-1}
 \frac{2(j+\theta)}{(j+1)(2r+j)}.
\]
For $p\geq q\geq r\geq2$, put
\[
\kappa_{p,r}
 =
 \max\left\{
 \frac{2}{\sqrt r},
 \frac1{\sqrt r}
 +\min\left\{\frac2{\sqrt p},\frac13\right\}
 \right\}, \quad 
\rho_{p,q}
 =
 \max\left\{
 \frac2{\sqrt p},
 \frac1{\sqrt p}
 +\min\left\{\frac2{\sqrt q},\frac13\right\}
 \right\}.
\]
Finally, let
\begin{align*}
\eta_0(p,q,r)
 &=
 \frac{
 \displaystyle
 \frac1p+\frac1q-
 \left(
 \frac1r-\frac1{qr}+\frac{\kappa_{p,r}}q
 \right)}
 {p+q+1+1/r},\\
\eta_1(p,q,r)
& =
 \frac1{p+1}
 \min\left\{
 \frac1p+\frac{1-\rho_{p,q}}q-\Lambda_r(0),
 \frac1p+\frac1q-\Lambda_r(\rho_{p,q})
 \right\}.
\end{align*}
We are now ready to prove the following result. 

\begin{corollary}\label{cor:general-campana-bound}
Let $\ve>0$ and let $p\geq q\geq r\geq2$. Then
\[
N(B)
 \ll_{\varepsilon,p}
 B^{1/p+1/q-\eta(p,q,r)+\varepsilon},
\]
where
$\eta(p,q,r)
 =
 \max\{0,\eta_0(p,q,r),\eta_1(p,q,r)\}$.
\end{corollary}

\begin{proof}
The upper bound 
$N(B)
 \ll_{p}
 B^{1/p+1/q}$ follows from~\eqref{eq:triv}.
We next apply Theorem~\ref{thm:main2} directly to
\[
\lambda x^p+\mu y^q=\nu z^{r+k},
\]
for $0\leq k\leq r-1$. Putting $M=\max\{X,Y\}$, we see that
$W\leq M^{1/\sqrt r}$ in~\eqref{eq:W}. Hence Theorem~\ref{thm:main2}
gives
\[
\#\left\{(x,y,z)\in \mathbb N^3:
\begin{array}{l}
x\le X,\ y\le Y,\ z\le Z,\ \gcd(x,y,z)=1,\\
\lambda x^p+\mu y^q=\nu z^{r+k}
\end{array}
\right\}
\ll_{\varepsilon,p}
H^{\ve}\max\{X,Y\}^{\kappa_{p,r}},
\]
where $H=\max\{X,Y,Z\}$,
since the term involving
$M^{1/(r+k)}$ is absorbed by $M^{2/\sqrt r}$.
It now follows from Theorem~\ref{prop:etagen} that 
\[
N(B)
 \ll_{\varepsilon,p}
 B^{1/p+1/q-\eta_0(p,q,r)+\varepsilon}.
\]

We now obtain a second estimate by permuting the variables and
applying Theorem~\ref{thm:main2} with $x$ as the distinguished
variable. Uniformly for $0\leq k\leq r-1$, this gives
\[
\#\left\{(x,y,z)\in \mathbb N^3:
\begin{array}{l}
x\le X,\ y\le Y,\ z\le Z,\ \gcd(x,y,z)=1,\\
\lambda x^p+\mu y^q=\nu z^{r+k}
\end{array}
\right\}
\ll_{\varepsilon,p}
H^{\ve}\max\{Y,Z\}^{\rho_{p,q}}.
\]
We now rework the proof of 
Theorem~\ref{prop:etagen} with this alternative input. Thus $S,T,U$ denote the
sums of the dyadic exponents, as in~\eqref{eq:UVW}. 
Let
\[
\gamma=\max_{0\leq k\leq r-1}\gamma_k
\]
and choose $k$ such that $\gamma_k=\gamma$. After fixing every
variable except $x_0,y_0,z_k$, the preceding estimate gives
\[
\Delta
 \leq
 S+T+U-\alpha_0-\beta_0-\gamma
 +\rho_{p,q}\max\{\beta_0,\gamma\}
 +\varepsilon.
\]

Suppose first that $\gamma\leq\beta_0$. Using
\[
S-\alpha_0\leq1-pS,\qquad
T-\beta_0\leq1-qT,\qquad
\beta_0\leq\frac1q,
\]
we obtain
$\Delta
 \leq
 2-pS-qT+(U-\gamma)+\frac{\rho_{p,q}}q+\varepsilon$.
Since
\[
pS+qT
 =
 p(S+T)-(p-q)T
 \geq
 p(S+T)-\frac{p-q}{q}
\]
and $\Delta\leq S+T$, by~\eqref{eq:triv'}, 
it follows from Lemma~\ref{lem:gammarho}
that
\begin{align*}
\Delta
 &\leq
 \frac{
 2+\Lambda_r(0)+(p-q+\rho_{p,q})/q
 }{p+1}
 +\varepsilon\\
 &=
\frac1p+\frac1q
 -
 \frac1{p+1}
 \left(
 \frac1p+\frac{1-\rho_{p,q}}q-\Lambda_r(0)
 \right)
 +\varepsilon.
\end{align*}

Suppose now that $\beta_0<\gamma$. In this case,
\[
\Delta
 \leq
 2-pS-qT+(U-\gamma+\rho_{p,q}\gamma)+\varepsilon.
\]
The same argument and Lemma~\ref{lem:gammarho} combine to give
\begin{align*}
\Delta
& \leq
 \frac{
 2+\Lambda_r(\rho_{p,q})+(p-q)/q
 }{p+1}
 +\varepsilon\\
 &=
\frac1p+\frac1q
 -
 \frac1{p+1}
 \left(
 \frac1p+\frac1q-\Lambda_r(\rho_{p,q})
 \right)
 +\varepsilon.
\end{align*}
Thus every dyadic box contributes at most
\[
\ll_{\ve,p}B^{1/p+1/q-\eta_1(p,q,r)+\varepsilon}.
\]
The number of boxes is a power of $\log B$, which is absorbed
into $B^\varepsilon$, and which thereby completes the proof.
\end{proof}

\begin{proof}[Proof of Theorem~\ref{t:uv-main}]
For fixed integers $u\geq v\geq0$, we have 
\[
\rho_{r+u,r+v}
 =
 \frac3{\sqrt r}+O_{u,v}(r^{-3/2}).
\]
The function $\phi_\theta(t)= \frac{2(t+\theta)}{(t+1)(2r+t)}$ has its stationary point at 
$t=-\theta+\sqrt{(1-\theta)(2r-\theta)}$. Thus, uniformly for $\theta=O(r^{-1/2})$, the maximising integer is $t=\sqrt{2r}+O(1)$, and substitution gives
$\Lambda_r(\theta)
 =
 \frac1r+O(r^{-3/2})$.
It follows that each of the two quantities inside the minimum
defining $\eta_1(r+u,r+v,r)$ is
$\frac1r+O_{u,v}(r^{-3/2})$.
Since $p+1=r+O_{u,v}(1)$, we conclude that
\[
\eta_1(r+u,r+v,r)
 =
 \frac1{r^2}+O_{u,v}(r^{-5/2}),
\]
which is positive for all sufficiently large $r$. The result now follows from 
Corollary~\ref{cor:general-campana-bound}.
\end{proof}

\begin{remark}\label{rem:rem'}
For $p=r+2$ and $q=r+1$, one has
$\eta_0(p,q,r)>0$ for every $r\geq6$. For the remaining values,
putting $\eta_i=\eta_i(r+2,r+1,r)$,
direct calculation gives
\[
\begin{array}{c|cc}
 r & \eta_0 & \eta_1 \\ \hline
 2 & <0 & \frac1{100} \\[4pt]
 3 & <0 & 
 \frac{17}{360}-\frac{\sqrt5}{60} \\[6pt]
 4 & <0 & 
 \frac{38}{1155}-\frac{\sqrt6}{105} \\[6pt]
 5 & <0 & 
 \frac{53}{2184}-\frac{\sqrt7}{168}
\end{array}
\]
In particular, $\eta(p,q,r)>0$ for every $r\geq2$ and our bound for $N(B)$ is
always non-trivial.
\end{remark}

\end{document}